\documentclass[a4paper,english, 12pt]{article}
\usepackage{babel} 
\usepackage[ansinew]{inputenc}
\usepackage{amsmath, amssymb, amsthm}
\usepackage{nicefrac}
\usepackage[arrow, matrix, curve]{xy}
\usepackage[nospace,noadjust]{cite}
\usepackage{graphics}
\usepackage{graphicx}
\usepackage[lmargin=3.00cm,rmargin=3.00cm,tmargin=3cm,bmargin=3cm]{geometry}

\newcommand{\R}{\mathcal{R}}
\newcommand{\Reals}{\mathbb{R}}

\newcommand{\Z}{\mathbb{Z}}
\newcommand{\Q}{\mathbb{Q}}

\newcommand{\C}{\mathbb{C}}

\newcommand{\F}{\mathcal{F}}

\newcommand{\T}{\pmb{T}}

\renewcommand{\P}{\mathcal{P}}

\newcommand{\id}{\textnormal{id}}

\newcommand{\Hom}{\textnormal{Hom}}

\newcommand{\colim}{\textnormal{colim}}

\newcommand{\into}{\hookrightarrow}
\newcommand{\tr}{\textnormal{tr}}

\newcommand{\Ind}{\textnormal{Ind}}

\renewcommand{\S}{\mathcal{S}}

\newcommand{\hocolim}{\textnormal{hocolim}}

\renewcommand{\T}{\mathcal{T}}
\newcommand{\sing}{\textnormal{sing}}
\newcommand{\hofib}{\textnormal{hofib}}
\newcommand{\cof}{\rightarrowtail}

\renewcommand{\Re}{\textnormal{Re}}

\renewcommand{\tilde}{\widetilde}
\renewcommand{\bar}{\overline}

\newtheorem{Lemma}{Lemma}[section]
\newtheorem{Theorem}[Lemma]{Theorem}
\newtheorem*{Theorem*}{Theorem}
\newtheorem{Proposition}[Lemma]{Proposition}

\theoremstyle{definition}

\newtheorem{Definition}[Lemma]{Definition}

\theoremstyle{remark}
\newtheorem{Remark}[Lemma]{Remark}

\title{Twisted Higher Smooth Torsion}
\author{Chris Ohrt}

\begin{document}

\maketitle

\begin{abstract} In this paper we extend Badzioch's, Dorabiala's, and Williams' definition \cite{zbMATH05551153} of cohomological higher smooth torsion to a twisted cohomological higher torsion invariant. Additionally, we show that this still satisfies geometric additivity and transfer, and will also satisfy additivity and transfer for coefficients.
\end{abstract}  

\tableofcontents

\section{Introduction}

Smooth parametrized torsion was introduced by Dwyer, Weiss, and Williams in \cite{zbMATH02115779}. Its cohomological version was defined by Dorabiala, Badzioch, and Williams in \cite{zbMATH05551153}. Using homotopical techniques they construct  characteristic classes $\tau_{2k}(E)\in H^{4k}(B)$ for any smooth bundle $E\to B$ that distinguishes between most smooth structures on $E.$ There are many constructions (as introduced by Wagoner, Klein, Igusa, Bismut, Lott, Goette and many others (\cite{Igusa2}, \cite{0793.19002}, \cite{0837.58028},  \cite{1071.58025})) like this and they are commonly referred to as (untwisted) higher torsion invariants. The advantage of this particular construction compared to other higher torsion invariants is that it is very natural, but unfortunately it does not lend itself to any calculations. 

On the other hand, Igusa and Klein used parametrized Morse theory to define Igusa-Klein torsion and were able to do basic calculations with it \cite{Igusa2}. Furthermore, Igusa defined a system of three axioms and showed that the space of untwisted higher torsion invariants (satisfying these axioms) is 2-dimensional and spanned by Igusa-Klein torsion and essentially a Chern class \cite{Igusa1}. This not only helps in the comparison of different torsion invariants but also enables more advanced calculations since the axioms give techniques for vertically deconstructing a smooth bundle into easier pieces. Badzioch, Dorabiala, Klein, and Williams showed that smooth parametrized torsion satisfies Igusa's axioms and is in fact a non-zero multiple of Igusa-Klein torsion \cite{zbMATH05848700}.

Besides untwisted Igusa-Klein torsion, there also exists a twisted version which takes a smooth bundle $E\to B$ together with a local system $\F\to E$ and constructs a characteristic class $\tau_k(E;\F)\in H^{2k}(B;\Reals)$ (twisted torsion classes exist in every other degree, while untwisted ones only exist in every 4th degree). The author expanded Igusa's axioms to an axiom system for twisted higher torsion invariants and showed that the space of twisted higher torsion invariants is again spanned by the (twisted) Igusa-Klein torsion and a Chern class \cite{Axioms}. This paper defines a twisted version of the cohomological smooth higher torsion and verifies that it satisfies most of the axioms for higher twisted torsion. This reduces the problem of relating the twisted smooth parametrized torsion to the twisted Igusa-Klein torsion to the task of calculating the smooth parametrized torsion of a universal $S^1$-bundle.

\subsection*{Overview}

The first section recaps constructions which have essentially been done in \cite{zbMATH05551153} already: using partitions we give an explicit model for $Q(X_+)\simeq \Omega^\infty\Sigma^\infty X_+$ if $X$ is a compact manifold. For a smooth bundle $E\to B$ we then define maps
	$$\xymatrix{
		&	Q(E_+)\ar[d]^{\lambda_\F}\\
		B\ar[ur]^{p^!}\ar[r]_{c_\F}	&	K(\C)}
	$$
where the above diagram commutes up to preferred homotopy. We show that $c_\F$ can be contracted to a constant map if $\F$ is unipotent (Definition \ref{unipotent}). Consequently we get a lift 
	$$\tau_\F:B\to Wh_\F(E):=\hofib(Q(E_+)\to K(\C))$$
which we call the torsion map. This was essentially defined by Dwyer, Weiss, and Williams already \cite{zbMATH02115779}. 

In the untwisted case, one uses the map $Q(E_+)\to Q(S^0)$ to pull back a universal class $b_{2k}\in H^{4k}(Wh(*);\Reals)$ (this class is derived from the Borel regulator \cite{zbMATH03495395} $b_{2k}\in H^{4k-1}(K(\Q);\Reals)$) all the way to $H^{4k}(B;\Reals).$ For the twisted torsion, we cannot directly mirror this process, as there cannot be a nontrivial local system on the point $*.$ So we need to replace the point with the ``$G$-equivariant point'' $BG$ (where $\rho:G\to U(m)$ is the representation corresponding to the local system $\F.$) Since the standard model for the classifying space $BG$ is not a compact manifold, we need to use a compact manifold approximation to $BG.$ All this is done in the second section.

Finally we are left to verify which of the axioms introduced by Igusa and the author are satisfied by the twisted smooth higher torsion. The seven axioms are naturality, triviality, geometric additivity, geometric transfer, additivity for coefficients, transfer for coefficients, and continuity (they are stated explicitly later). It is immediate from the definition that twisted smooth higher torsion behaves naturally under pull backs along base maps $B\to B'.$ It is also easy to see that $\tau_{2k+1}(E;\F)=0$ if $\F$ is trivial; this follows because if $\F$ is trivial, the whole construction factorizes through $K(\Q)$ and the Borel regulators vanish on $H^{4k+2}(K(\Q);\Reals).$ In this paper we will show:

\begin{Theorem} Twisted higher smooth torsion satisfies geometric additivity and transfer as well as additvity and transfer for coefficients.
\end{Theorem}

While geometric addivity and geometric transfer have already been done in the non-twisted case in \cite{zbMATH05848700} (the proof is recalled and adapted), the proofs for additivity and transfer for coefficients are original. 

In the end we are left to prove the continuity axiom. Let $S^\infty\to \mathbb{C}\mathbb{P}^\infty$ be the universal $S^1$-bundle. Since the cyclic group of order $n$ acts on $S^\infty$ we can define $\F_\xi\to S^\infty/n$ to be the local system corresponding to the $n$th-root of unity $\xi.$ The continuity axiom essentially states that the dependence of the class $\tau(S^\infty/n;\F_\xi)$  is continuous on $\xi\in \Q/\Z.$ Igusa proved that Igusa-Klein torsion satisfies this axiom by explicitly calculating the torsion of these universal bundles \cite{Igusa2}. Providing such a calculation for smooth torsion is part of the ongoing work of the author.

\section{The higher smooth torsion map} 

\subsection{The manifold approach}

This section repeats the constructions made in the beginning of \cite{zbMATH05551153}. Let $X$ be a compact manifold. We will define a model for $\Omega^\infty\Sigma^\infty X_+$ which (by abusing notation a bit) we will call $Q(X_+).$ It will be constructed as the direct limit under stabilization of the Waldhausen K-theory spaces of certain categories of partitions. We will refrain from giving details as they can be found in \cite{zbMATH05551153}.

A partition of $X\times I$ is a (not necessarily smooth and possibly with corners) codimension 0 submanifold $M\subset X\times I$ that represents the lower half of a division of the interval in two parts and is somewhat standard around the boundary and on the lower third. Part of the data is also a vector field transversal to the boundary which can be used to smoothen the partition. The set  $\P_k(X)$ consists of partitions of $X\times I$ parametrized over $\Delta^k.$ These fit together in a simplicial set $\P_\bullet(X).$ There is a stabilization map $\P_\bullet (X)\to \P_\bullet (X\times I)$ defined by putting the non-trivial part of a partition of $X\times I$ into the middle third (of the second interval) to get a partition of $(X\times I)\times I.$ We note that there is a partial monoid structure on $\P_\bullet (X)$ where we add two partitions of $X\times I$ if they do not share any non-trivial parts. Stabilization now provides a monoid structure on $\colim_n \P_\bullet(X\times I^n).$ 

The sets $\P_k(X)$ can also be viewed as partially ordered sets  by inclusion, and hence as categories. So we can apply the Waldhausen $S$-construction (or rather the Thomason variant thereof) \cite{zbMATH03927168} to get bisimplicial categories $\T_\bullet \P_\bullet (X).$ Recall that the objects of the cageory $\T_n\P_0$ are $(n+1)$-tuples of partitions $(M_i)_{i=0}^n$ with
	$$M_0\subset M_1\subset\ldots\subset M_n.$$
The space $\T_0\P_\bullet (X)$ turns out to be contractible and hence we get a weak equivalence
	$$|\T_\bullet\P_\bullet (X)|\simeq |\T_\bullet\P_\bullet (X)|/|\T_0\P_\bullet(X)|$$
and by abuse of notation we will denote the right side by $|\T_\bullet \P_\bullet(X)|.$ Note that this is now endowed with a canonical base point.

By stabilization we get a space 
	$$Q(X_+):=\Omega\hocolim_n|\T_\bullet \P_\bullet(X\times I^n)|\simeq \Omega^\infty\Sigma^\infty X_+.$$
The weak equivalence on the right is rather intricate and was shown by Waldhausen in \cite{zbMATH03958257} and \cite{zbMATH04163135}.

The next goal is to get a map from $Q(X_+)$ to a model of the Waldhausen K-Theory $A(X).$ Details on the following construction can still be found in \cite{zbMATH05551153}. First we recall the standard model of $A(X)$:  Let $\R^{fd}(X)$ be the category of finitely dominated retractive spaces over $X.$ Then we set -- again using Thomason's model for the Waldhausen $S$-construction
	$$A(X):=\Omega(|\T_\bullet \R^{fd}(X)|/|\T_0 R^{fd}(X)|).$$

In order to get the map $Q(X_+)\to A(X)$ we need to ``thicken up'' $A$ a bit: we let $R^{fd}_\bullet(X)$ be the simplicial category with $k$th level being the retractive spaces $Y$ over $X$ parametrized over $\Delta^k.$ Again we have a stabilization map $\R^{fd}_\bullet(X)\to \R^{fd}_\bullet(X\times I)$ where ``the interesting part happens in the middle third.'' We then form 
	$$A_p(X):=\Omega\hocolim_n|\T_\bullet \R^{fd}_\bullet(X\times I^n)|$$
and since every partition $M$ can be viewed as a retractive space, there is a map
	$$\alpha:Q(X_+)\to A_p(X).$$

The difference between the models $A_p(X)$ and $A(X)$  is that the former is endowed with an additional simplicial enrichment and a stabilization process. Clearly $A(X)$ includes into $A_p(X)$ as the 0-simplices of the bottom of the colimit, giving a map 
	$$i:A(X)\to A_p(X)$$
which Waldhausen shows explicitly to be a weak homotopy equivalence \cite{zbMATH03958257}.

It will be convenient for us to have a map to the standard model of Waldhausen A-theory and hence we define a space $\tilde Q(X_+)$ as the homotopy pull back	
	$$\xymatrix{
		\tilde Q(X_+)\ar[r]\ar[d]_{\tilde \alpha}	&	Q(X_+)\ar[d]^\alpha\\
		A(X)\ar[r]^i	&	A_p(X).}
	$$
Obviously, we have a weak equivalence $\tilde Q(X_+)\simeq Q(X_+).$

\begin{Remark} \label{mapsinto} The following helps greatly in defining maps into $Q(X_+)$ (and $A_p(X)$ and $\tilde Q(X_+)$). Recall that there is a natural map
	$$|\T_1\P_\bullet(X)|\times \Delta^1\to |\T_\bullet\P_\bullet(X)|$$
given by the inclusion of the 1-skeleton in the $\T_\bullet$ direction \cite{zbMATH03927168}. After stabilizing and taking the adjoint this gives a map 
	$$|\T_1\P_\bullet(X)|\to \Omega Q(X_+).$$
Hence it is always enough to define a functor $\mathcal{C}\to\T_1\P_\bullet(X)$ to get a map $|\mathcal{C}|\to Q(X_+)$ for any small category $\mathcal{C}.$
\end{Remark}

\subsection{The transfer map}

Let $p:E\to B$ be a bundle of compact manifolds. The goal of this section is to give a model for the Gottlieb-Becker transfer $p^!:B\to Q(E_+).$ We again follow \cite{zbMATH05551153}
 very closely and will refer to this for any details.

First let $\sing B$ denote the simplicial set of smooth simplices in $B$ viewed as a simplicial category with no non-trivial morphisms. We note that $|\sing B|\simeq B.$ We will explicitly construct a map $p^!:|\sing B|\to \tilde Q(E_+).$ We will do this in two steps:

First we define a map $Q(p^!):Q(B_+)\to Q(E_+).$ This is essentially done by pulling back partitions $M\mapsto (p\times id)^{-1}(M)$ and the attached smoothing data. Notice that pulling back the smoothing data depends on the smooth structure of the bundle. In fact this is the only piece of the construction that will depend on the smooth data. It easily gives rise to a map
	$$\tilde Q(p^!):\tilde Q(B_+)\to \tilde Q(E_+).$$

Then we can explicitly define a map $|\sing B|\to \tilde Q(B_+).$ Details can be found in \cite{zbMATH05551153}. The core of this is the following: Let $\sigma:\Delta^n\to B$ be a simplex. The exponential map gives an inclusion $\sigma^*TB_\epsilon\times I\to B\times \Delta^n\times I,$ where $TB_\epsilon$ denotes a small tangential tube. After a slight modification this indeed represents a parametrized partition of $B.$

Now composition of those two maps gives a map
	$$p^!:|\sing B|\to \tilde Q(E_+)$$
which agrees with the Becker-Gottlieb transfer \cite{zbMATH05551153}.

\begin{Remark} It is worth noting that $\tilde \alpha p^!:|\sing B|\to A(E_+)$ comes from the functor sending $\sigma:\Delta^n\to B$ to the retractive $E_+\to E_+\sqcup \sigma^*E\to E_+.$ Compare to \cite{zbMATH05551153}.
\end{Remark}

\subsection{Linearization}

We still follow \cite{zbMATH05551153} closely to define linearization maps. Let $R$ be a ring and let $\P Ch^{fd}(R)$ be the Waldhausen category of homotopy finitely dominated chain complexes of projective $R$-modules. Recall that the Waldhausen K-theory of this category is just a model for the algebraic K-theory $K(R)$ of $R$ \cite{zbMATH03927168}.

Now let $X$ be a compact manifold and $\F$ a local system of $R$-modules on $X.$ Then we get a functor
	$$\R^{fd}(X)\to \P Ch^{fd}(R)$$
by sending a retractive space $X\to Y\to X$ to the relative singular chain complex $C_*(Y,X;\F).$ This induces a linearization map
	$$\lambda^\R_\F:A(X)\to K(R)$$
and if we compose with the assembly $\alpha:\tilde Q(X_+)\to A(X)$ we get a map 
	$$\lambda_\F:\tilde Q(X_+)\to K(R).$$
	
Let $E\to B$ be a bundle of compact manifolds and let $\F$ be a local system of $R$-modules on $E.$ Similarly to before we can define a functor
	$$\sing B\to w\P Ch^{fd}(R)$$ 
(the $w$ indicates that we are only looking at quasi-isomorphisms as morphisms). In particular, this functor sends a simplex $\sigma:\Delta^k\to B$ to the chain complex $C_*(\sigma^*E, \F).$ Using Remark \ref{mapsinto} this gives rise to a map
	$$c_\F: |\sing B|\to K(R).$$
	
By analyzing the concrete map descriptions it can be shown \cite{zbMATH05551153}

\begin{Theorem} Let $E\to B$ be a bundle of compact manifolds, $R$ a ring, and $\F$ a local system of $R$-modules on $E.$ Then there is a preferred homotopy which makes the following diagram commute:
	$$\xymatrix{
		&	\tilde Q(E_+)\ar[d]^{\lambda_\F}\\
		|\sing B|\ar[ur]^{p^!}\ar[r]_{c_\F}	&	K(R)}$$
	The homotopy is induced by the isomorphism $H_*(\sigma^*E;\F)\cong H_*(E\sqcup \sigma^* E, E;\F).$
\end{Theorem}

This will be the starting point for us to define smooth parametrized torsion.

\subsection{Unreduced and reduced smooth parametrized torsion}

This section is where we depart slightly from \cite{zbMATH05551153} in that our results will be a little bit more general than there. The idea here is that if we can show that the map $c_\F:|\sing B|\to K(R)$ is homotopic to the constant map with value the 0 complex $0\in K(R)$ then we get a lift 
	$$|\sing B|\to \hofib\left(\tilde Q(E_+)\stackrel{\lambda_\F}{\longrightarrow} K(R)\right)_0=:Wh_\F(E)$$
where we call the codomain the Whitehead space of $E.$ This will not always be the case, but the following condition is almost sufficient:

\begin{Definition}\label{unipotent} Let $E\to B$ be as before and let $\F$ be a complex local system on $E.$ Let $B$ be connected, $b_0\in B$ be the basepoint, and let $F$ be the fiber over $b_0.$ We say $\pi_0 B$ acts unipotently on $H_*(F;\F)$ if there exists a filtration of $H_*(F;\F)$ by $\pi_1 B$ submodules 
	$$0=V_0(F)\subset \ldots \subset V_k(F)=H_*(F;\F)$$
such that $\pi_1 B$ acts trivially on the quotients $V_i(F)/V_{i-1}(F).$
\end{Definition}

\begin{Theorem}\label{contraction} Let $E\to B$ be a bundle, $B$ path-connected, $b_0\in B$ the basepoint, $F_{b_0}$ the fiber over the basepoint,  $\F\to E$ a complex local system such that $\pi_1 B$ acts unipotently on $H_*(F,\F).$ Then there exists a preferred homotopy from the map $c_\F:|\sing B|\to K(\C)$ to the constant map with value the complex $H_*(F_{b_0},\F)\in K(\C)$ (with trivial boundary maps).
\end{Theorem}

\begin{Remark} While this theorem is never directly stated in \cite{zbMATH05551153}, it is remarked that the theorems there can be generalized (even further than we need it) to include this version. 
\end{Remark}

Let $k:|w\P Ch^{fd}(\C)|\to K(\C)$ be the inclusion. We first prove the following lemma:

\begin{Lemma} Let $H:w\P Ch^{fd}\to w\P Ch^{fd}$ be the functor that assigns to each complex its homology complex. Then there is a preferred homotopy 
	$$k\simeq k\circ |H|$$
\end{Lemma}

\begin{proof} This is (even a little more generally) proved in \cite{zbMATH05551153}. The idea is to define a complex $P_q C$ for a complex $C$ by cutting of $C$ at $C_q$ and setting $(P_q C)_{q+1}:=\partial C_{q+1}.$ Then if $Q_qC$ is the kernel $P_qC\to P_{q-1}C$ it is a complex homotopy equivalent to its homology complex which has only one non-zero entry being $H_q(C).$ Now we use Waldhausen's additivity theorem repeatedly and we are done.
\end{proof}

\begin{proof}[Proof of the Theorem]First of all the lemma gives us a preferred homotopy from $c_\F$ to the map constructed by the functor $\sing B\to w\P Ch^{fd}(\C)$ with $\sigma\mapsto H_*(\sigma^*E;\F).$ Now there is an isomorphism
	$$H_*(F_{\sigma_0};\F)\cong H_*(\sigma^* E; \F),$$
where $\sigma_0$ denotes the base point of $\sigma.$ This now gives a homotopy from $c_\F:|\sing B|\to K(\C)$ to the map defined by the functor
	$$H^0_\F:\sigma\mapsto H_*(F_{\sigma_0};\F)$$
We can summarize this as $c_\F\simeq k\circ |H^0_\F|.$ 

Now we have a filtration
	$$0=V_0(F_{b_0})\subset V_1(F_{b_0})\subset \cdots \subset V_n(F_{b_0})=H_*(F_{b_0};\F)$$
such that $\pi_1 B$ acts trivially on the subquotients (and all the $V_i(F_{b_0})$ are $\pi_1 B$ modules.) This is exactly what it means for $\pi_1 B$ to act unipotently. Since $B$ is path connected, there is an isomorphism $f_b:F_{b_0}\to F_b$ and we set $V_i(F_b):=f^*_b(V_i(F_{b_0}))$. This map $f_b:F_{b_0}\to F_b$ might not be canonical but because of the $\pi_1 B$ equivariance of the original filtration the resulting filtration
		$$0=V_0(F_{b})\subset V_1(F_{b})\subset \cdots \subset V_n(F_{b})=H_*(F_{b},\F)$$
does not depend on the choices made. Now we look at the functor $v:\sing B\to w\P Ch^{fd}(\C)$ given by
	$$v(\sigma):=\bigoplus V_i(F_{\sigma_0})/V_{i-1}(F_{\sigma(0)}).$$
Waldhausen's additivity theorem gives a homotopy $k\circ |H^0_\F|\simeq k\circ |v|.$ Overall we have
	$$c_\F\simeq k\circ |v|.$$

By the unipotence assumption we get that the induced isomorphism
	$$V_i(F_{\sigma(0)})/V_{i-1}(F_{\sigma(0)})\cong V_i(F_{b_0})/V_{i-1}(F_{b_0})$$
does not depend on any choices (of the path lifting used to define the isomorphism) and hence we get a natural isomorphism between the functor $v$ and the functor $v_{b_0}:\sing B\to w\P Ch^{fd}(\C)$ given by $v_{b_0}(\sigma):=\oplus V_i(F_{b_0})/V_{i-1}(F_{b(0)}).$ So we have a homotopy $k\circ |v|\simeq k\circ |v_{b_0}|$ and conclusively a homotopy
	$$c_\F\simeq k\circ |v_{b_0}|.$$
Now an application of the additivity theorem gives the preferred homotopy between $c_\F$ and the constant map to $H_*(F_{b_0},V)\in K(\C).$ We shall call this homotopy $\omega_\F.$
\end{proof}

\begin{Definition} Let $p:E\to B$ be a compact manifold bundle with $B$ connected. Let $F_{b_0}$ be the fiber over the basepoint and let $\F$ be a unipotent complex local system over $E.$ We view the homology complex $H_*(F_{b_0};\F)$ as an element in $K(\C)$ and we define the unreduced Whitehead space
	$$Wh_\F(E,b_0):=\hofib(\tilde Q(E_+)\stackrel {\lambda_\F}\to K(\C))_{H_*(F_{b_0};\F)}.$$
The unreduced smooth torsion of $p$ is the map $\tilde \tau_\F:|\sing B|\to Wh_\F(E,b_0)$ determined by the transfer $p^!$ and the homotopy $\omega_\F.$
\end{Definition}

We want to make this independent of the basepoint choice. The answer is the reduced torsion:

\begin{Definition} For a compact manifold bundle $p:E\to B$ with base point $b_0\in B$  and unipotent complex local system $\F$ on $E$ we define the Whitehead space 
	$$Wh_\F(E):=\hofib(\tilde Q(E_+)\to K(\C))_0.$$
The reduced smooth torsion $\tau_\F(p)$ is the map $|\sing B|\to Wh_\F(E)$ obtained from $p^!$ by subtracting the element $p^!(b_0)\in \tilde Q(E_+)$ from the map $p^!$ and the path $\omega_{\F|{b_0}\times I}$ from the contracting homotopy $\omega_\F.$
\end{Definition}

\section{Cohomological Torsion}

Let $p:E\to B$ be a bundle of compact manifolds and $\F$ a finite, unipotent, complex local system. By definition $\F$ is classified by a map
	$$\phi: E\to BG$$
where $G$ is a finite group. We will use this map to pull back a universal cohomological class along our torsion maps. Before we can do this, we need to introduce a compact approximation model to $BG:$

More precisely, let $G$ be a group and let $\rho: G\into U(m)$ be an inclusion corresponding to a faithful unitary action of $G$ on $\C^m.$ We will describe high dimensional manifolds with boundary $EG(N)$ each equipped with a free $G$-action and $G$-equivariant connecting inclusions $EG(N)\into EG(N+1).$ We will see that the $EG(N)$ are highly connected meaning that the colimit
	$$EG:=\colim_N EG(N)$$
is contractible. So if we set $BG(N)$ to be the quotient $BG(N):=EG(N)/G$ we see that
	$$BG:=\colim_N BG(N)/G\cong EG/G$$
is a classifying space for the group $G$ with universal bundle $EG\to BG.$

\subsection{The Set $EG'(N)$}

In this section, we will define a set $EG^{coarse}(N)$ on which $G$ acts freely. Again, let $G$ be a group and $\rho:G\into U(m)$ an embedding. Since $G$ acts on $\C^m$ it will also act on the vector space $\Hom(\C^N,\C^m)\cong \C^{Nm}$ and we have:

\begin{Lemma} If $G$ acts faithfully on $\C^m,$ it acts freely on the subset $\Theta(N)\subset \Hom(\C^N,\C^m)$ of surjective maps.
\end{Lemma}

\begin{proof} The fact that $G\subset U(m)$ guarantees that the action of $G$ preserves the rank of the linear transformations it acts on and hence it restricts to an action on $\Theta(N).$ Assume we have $h,g\in G$ and $f\in \Theta(N)$ with $h\neq g.$ Since the action on $\C^m$ is faithful, there must be a $y\in \C^m$ with $hy\neq gy.$ So we find an $x\in \C^N$ with $f(x)=y$ and we see $hf\neq gf.$
\end{proof}

Now let $S^{2Nm-1}\subset \C^{Nm}$ be the unit sphere and let $\Phi(N)\subset \Hom(\C^N,\C^m)\cap S^{2Nm-1}$ be the set of non-surjective maps. We define
	$$EG^{coarse}(N):=\Theta(N)\cap S^{2Nm-1}= S^{2Nm-1}\backslash \Phi(N).$$
As the action of $G$ is unitary and free on $\Theta(N),$ it restricts to a free action on $EG^{coarse}(N).$ Unfortunately, this is not a compact manifold. To overcome this, we will analyze the structure of $\Phi(N)$ (it is a stratified space) and then coherently thicken it up to have a codimension 0 smooth submanifold with boundary $\Phi^{smooth}(N)\subset S^{2Nm-1}$ and eventually define $EG(N):=S^{2Nm-1}\backslash \Phi^{smooth}(N)\cup\partial\Phi^{smooth}(N).$

\subsection{The structure of the set $\Phi(N)$}

\begin{Proposition} The set $\Phi(N)\cap S^{2Nm-1}\subset S^{2Nm-1}$ is a stratified closed space with every stratum having codimension at least $2N-2m.$  Consequently, the space $S^{2Nm-1}\backslash \Phi(N)$ is $2N-2m$ connected (for $N$ large enough).
\end{Proposition}

\begin{proof} Let $\Sigma_N$ be the symmetric group on $N$ elements. This acts on $\C^{Nm}\cong (\C^m)^N$ by permutation and this action commutes with the action of $G.$ We can view $\Phi(N)$ as the set of $m\times N$-matrices of rank less than $m$ and hence we can write it as the union
	$$\Phi(N)=\Sigma_N \Phi_1(N)\cup \Sigma_N \Phi_2(N)\cup\cdots\cup \Sigma_N \Phi_{m-1}(N),$$
where $\Sigma_N \Phi_i(N)$ denotes the orbit under the $\Sigma_N$-action of the set 
	$$\Phi_i(N):=\{(v_1,v_2,\ldots, v_{i}, a_{11}v_1+\ldots+a_{1i}v_i,\ldots,a_{N-i,1}v_1,\ldots, a_{N-i,i}v_i)\}\cap S^{2Nm-1},$$
where the $v_i$'s are linearly independent vectors in $\C^m$ and the $a_{jl}\in \C$'s are any scalars. In other words, $\Sigma_N\Phi_i(N)$ is the set of $m\times N$ matrices of rank $i$ which are on the unit sphere in $\C^{Nm}.$

Notice that the points in $\Phi_i(N)$ are defined by the equation 
	$$\sum_{l=1}^i\left(1+\sum_{j=1}^{N-i}|a_{jl}|^2\right)|v_l|^2+\sum_{j=1}^{N-1}\sum_{l< l'}\Re(\bar{a_{jl}}a_{jl'}\langle v_l,v_{l'}\rangle)=1,$$
The gradient on the left (in the components of the $v_is$'s and the $a_{jl}$'s) can be calculated as
	$$\left(\begin{array}{c}
		\left(1+\sum_{j=1}^{N-i}|a_{j1}|^2\right)2v_1 + \sum_{j=1}^{N-i}\sum_{l'=2}^i a_{jl'}\bar{a_{j1}} v_{l'}\\
		\vdots\\
		\sum_{j=1}^{N-i}\sum_{l=1}^{i-1} a_{ji}\bar{a_{jl}}v_l+\left(1+\sum_{j=1}^{N-i}|a_{ji}|^2\right)2v_1\\
		\vdots

		\end{array} \right)
	$$
Here we view the vectors $v_l$ as real vectors in $\Reals^{2m}$ (though we still multiply them by complex numbers). Since the $v_i$ are complex linearly independent vectors the above gradient can never be 0. So $\Phi_i(N)$ is a submanifold of $S^{2Nm-1}\backslash \Phi_{i-1}(N).$

Notice that the intersection of $\Phi_i(N)$ with the hyperplanes given by fixing the $v_i$'s (at any admissable value) is closed whereas the intersection with the hyperplanes given by fixing the $a_{jl}$'s is neither open nor closed (in that hyperplane) unless $i=1$ in which  case it will be a closed 1-sphere. Analyzing this structure further, we can see that 
	$$\Sigma_N \Phi_1(N)\cup \Sigma_N \Phi_2(N)$$
is closed and inductively we see that $\Phi(N)$ must be closed where every point is contained in a stratum with dimension at most $2m(m-1)+2(N-m+1)(m-1).$ So every point is contained in a submanifold of $S^{2Nm-1}$ with codimension at least
	$$2Nm-1-2m(m-1)-2(N-m+1)(m-1)=2N-2m$$
\end{proof}

The next step is to carefully define a $U(m)$-equivariant tubular neighborhood of $\Phi(N)$ to form $\Phi^{thick}(N)$ which we then round off to get $\Phi^{smooth}(N).$ First notice that $U(m)$ acts on each of the $\Phi_i(N)$ and hence it will act on the normal bundles $\mathcal{N}\Phi_i(N)$ in $S^{2Nm-1}\backslash \Phi_{i-1}(N).$ Therefore we can find a closed tubular neighborhood $\Phi^{thick}_i(N)$ of $\Phi_i(N)$ such that $U(m)$ still acts on $\Phi^{thick}_i(N)$ and since the actions of $U(m)$ and $\Sigma_N$ commute, we can choose these tubular neighborhoods $\Sigma_N$-equivariantly. The union 
	$$\bigcup_{\sigma\in \Sigma_N} \sigma\Phi^{thick}_i(N)$$
might not be a manifold (since it may have corners), but it is possible to find neighborhoods of its corners on which $U(m)$ acts freely because the part $\Phi(N)$ on which $U(m)$ does not act freely is well away from the boundary of this union. Now we can round of the corners (as in \cite{Igusa1}) and since $U(m)$ is compact, we can average the rounding over the $U(m)$-action to get a manifold $\Phi^{smooth}_{i,\Sigma}(N)$ which still has a $U(m)$-action. 

We inductively take the unions
	$$\bigcup_{i=1}^l \Phi^{smooth}_{i,\Sigma}(N) \cup \Phi^{smooth}_{l+1,\Sigma}(N)$$
and by the same argument as above, we  can round off the corners $U(m)$-equivariantly, so that we eventually obtain the a smooth codimension 0 submanifold with boundary
	$$\Phi^{smooth}(N)\subset S^{2Nm-1}.$$
Analogously to $\Phi(N)$ being closed, we see that $\Phi^{smooth}(N)$ is also closed. It is furthermore obtained by forming tubular neighborhoods and rounding off corners and therefore it has the homotopy type of a codimension $2N-2m$ space. 

\begin{Definition} With the considerations above we set
	$$EG(N):=(S^{2Nm-1}\backslash B^{smooth}(N))\cup \partial B^{smooth}(N)$$
and we see that it is a $2N-2m$ connected manifold with boundary that has a free $U(m)$-action.
\end{Definition}

\subsection{The spaces $EG$ and $BG$}

Recall that $\Theta(N)$ denotes the subset of surjective maps in $\Hom(\C^N,\C^m).$ There clearly is an inclusion $\Theta(N)\into \Theta(N+1)$ (given by trivial extension to the $N+1$st component). This gives rise to an inclusion $S^{2Nm-1}\backslash \Phi(N)\into S^{2(N+1)m-1}\backslash \Phi(N+1).$ If we pay attention and round of corners coherently (by inductively ensuring that the rounding for $N+1$ agrees with the rounding for $N$ under the inclusion) we also get inclusions
	$$\iota :EG(N)\to EG(N+1)$$
and hence we can define

\begin{Definition} We set	
$$EG:=\colim_N EG(N).$$
\end{Definition}

We see immediately that $EG$ is contractible (it is an arbitrarily highly connected CW space) and has a free $U(m)$-action. So in particular it has a free $G$-action (for any group with an embedding $G\into U(m)$) and we see that the $G$-bundle
	$$EG\to EG/G=:BG$$
is universal. Notice that we also have
	$$BG\cong \colim EG(N)/G.$$

\subsection{Naturality of the construction}

First let $G\into U(m)\subset U(m+1)$ be a faithful representation of dimension $m$ that we can also view as a faithful representation of dimension $m+1$ which is trivial on the last component. We get an inclusion 
	$$\Hom(\C^N,\C^m)\into \Hom(\C^{N+1},\C^{m+1})$$
which is given by extending a homomorphism by the identity on the last component. Notice that this restricts to an inclusion between the surjective homomorphisms and if the smoothing is done correctly gives a $G$-equivariant inclusion
	$$EG(m,N)\into EG(m+1,N+1)$$
where the $m$-arguments indicate the dimension of the representation of $G$ that we consider. In the colimit this gives a map
	$$EG(m)\stackrel{\simeq}{\longrightarrow}EG(m+1)$$
and consequently $BG(m)\stackrel{\sim}{\to}BG(m+1).$

\begin{Remark} In the remainder of this paper we will always consider groups $G$ with the an implied choice of embedding $G\into U(m).$ We just saw that this is independent of the choice of the dimension $m$ of the embedding.
\end{Remark}

On the other hand, assume that we have two groups $G\into U(m)$ and $G'\into U(m').$ Then their product $G\times G'$ defines a faithful representation of dimension $m+m'.$ We have maps
	$$\Hom(\C^N,\C^m)\times \Hom(\C^N,\C^{m'})\to \Hom(\C^{2N},\C^{m+m'})$$
given in the obvious way. These preserve surjective homomorphisms and eventually give rise to maps
	$$EG(N)\times EG'(N)\to E(G\times G')(2N)$$
which induce a $G\times G$-equivariant homotopy equivalence
	$$EG\times EG'\simeq E(G\times G')$$
and consequently a homotopy equivalence $BG\times BG'\to B(G\times G').$

\subsection{Defining the cohomological torsion}

We return to the task at hand of defining cohomological higher torsion classes. For this let $E\to B$ be a smooth manifold bundle and $\F\to E$ a finite, complex local system on $E$ corresponing to a representation $\rho:\pi_1(E)\to G\into U(m)$ for a finite group $G.$ This also gives a map $\phi:E\to BG.$ Since $E$ is compact, there must be an $N$ such that $\phi$ restricts as
	$$E\to BG(N)\into BG.$$
By abuse of notation we will call the restriction $E\to BG(N)$ also $\phi.$ Now $\phi$ induces a map $\phi_Q:\tilde Q(E_+)\to \tilde Q(BG(N)_+).$ We get the following diagram

$$\xymatrix{
																								&	Wh_\F(E)\ar[d]\ar@{.>}[r]^-{\phi_{Wh}}						&	Wh_\F(BG(N))\ar[d]\\
		|\sing B|\ar[r]^{p^!}\ar[ur]^{\tau_\F(p)}		&	\tilde Q(E_+)\ar[r]^{\phi_Q}\ar[d]^-{\lambda_\F}	&	\tilde Q(BG(N)_+)\ar[d]^{\lambda_\F}\\
																								&	K(\C)\ar[r]^=																			&	K(\C).
		}
	$$

By inspection, the lower square commutes up to a preferred homotopy: The upper right composition corresponds to the functor that sends a retractive space $X$ over $E$ to the chain complex $C_*(X\sqcup_E BG(N), BG(N);\F)$ and the lower left corner comes from the functor that sends such an $X$ to the chain complex $C_*(X, E; \F).$ This preferred homotopy induces the map $\phi_{Wh}:Wh_\F(E)\to Wh_\F(BG(N)).$

While it is easy to see that $Q(BG(N)_+)\to Q(BG(N+1)_+)$ is a cofibration (it is essentially induced by a level-wise inclusion of simplicial sets), it is not clear if the pull-back $\tilde Q(BG(N))\to \tilde Q(BG(N+1))$ will since pull-back does not behave well with respect to cofibrations. Therefore, consider the following commutative diagram
	$$\xymatrix{
		\tilde Q(BG(1)_+)\ar[r]^{i_1}	&	\tilde Q(BG(2)_+)\ar[r]^{i_2}	&	\tilde Q(BG(3)_+)\ar[r]^{i_3}	&	\cdots\\
		\tilde Q(BG(1)_+)'\ar[u]^\simeq_{j_0}\ar@{>->}[r]^{i'_1}	&	\tilde Q(BG(2)_+)\ar[u]^\simeq_{j_1}\ar@{>->}[r]^{i'_2}	&	\tilde Q(BG(3)_+)'\ar[u]^\simeq_{j_2}\ar@{>->}[r]^{i'_3}	&	\cdots}
	$$
where the lower horizontal arrows are cofibrations so that the lower diagram is a cofibrant replacement of the upper diagram (we will see in Remark \ref{independence} that our constructions do not depend on the cofibrant replacement chosen). So we get
	$$\hocolim_N \tilde Q(BG(N)_+)\simeq \colim_N \tilde Q(BG(N)_+)'=:\tilde Q(BG_+).$$

Composition of $\lambda_\F$ with $j_N$ gives maps $\lambda'_\F:\tilde Q(BG(N)_+)'\to K(\C)$ with homotopy fiber $Wh'_\F(BG(N)).$ Together they define a map $\lambda'_\F:\tilde Q(BG_+)\to K(\C)$ and we call its homotopy fiber $Wh_\F(BG).$ After rationalization we get a diagram
	$$\xymatrix{
		Wh_\F(BG)_\Q\ar@{.>}[r]\ar[d]	&	\Omega K(\C)\ar[d]\\
		\tilde Q(BG_+)_\Q\ar[r]\ar[d]	&	{*}_\Q\ar[d]\\
		K(\C)\ar[r]^=	&	K(\C).
		}
	$$
Notice that $\tilde Q(BG(N)_+)$ is rationally highly connected, so we get that $\tilde Q(BG_+)$ is rationally contractible. This gives a preferred homotopy making the bottom square commute up to homotopy and giving the dotted map $Wh_\F(BG)\to \Omega K(\C)$ between homotopy fibers. 

There are the Borel regulators $b_i\in H^{2i+1}(K(\C);\Reals)$ defined in \cite{zbMATH03495395} which we can pull back along the adjoint of the map described above to get elements $b_i\in H^{2i}(Wh_\F(BG);\Reals).$ There also are maps $Wh'_\F(BG(N))\to Wh_\F(BG)$ and weak equivalences $Wh'_\F(BG(N))\to Wh_\F(BG(N))$ (the corresponding diagrams commute on the nose) and we can pull back $b_i$ further and use the weak equivalence to get an element $b_i\in H^{2i}(Wh_\F(BG(N));\Reals).$ Notice that these elements fit together well with respect to the inclusions $BG(N)\to BG(N+1).$

\begin{Definition} We define the higher twisted smooth torsion of the bundle $p:E\to B$ with coefficients $\F$ as the collection of cohomology classes
	$$\tau_i(E;\F):=\tau_\F(p)^*\phi^*_{Wh}(b_i)\in H^{2i}(B;\Reals).$$
\end{Definition}

If $N$ is large enough ($N>>\dim B$) this is well defined. 

\begin{Remark}\label{independence} The cohomology class $\tau_i(E;\F)$ does not depend on the cofibrant replacements $\tilde Q(BG(N)_+)'\to \tilde Q(BG(N)_+)$: let $\tilde Q(BG(N)_+)''\to \tilde Q(BG(N)_+)$ be another cofibrant replacement. Then we can form the pull back
	$$\tilde Q(BG(N)_+)''':=\tilde Q(BG(N)_+)'\times_{Q(BG(N)_+)}\tilde Q(BG(N)_+)''.$$
All constructions done behave naturally with respect to this pull-back (possibly combined with another cofibrant replacement) and hence the Borel classes on $\tilde Q(BG(N)_+)'$ and $\tilde Q(BG(N)_+)''$ pull back to a common class on $\tilde Q(BG(N)_+)'''$ and hence induce the same class on $|\sing B|.$
\end{Remark}

\section{Axiomatic torsion}

We will prove that the cohomological higher smooth twisted torsion defined above satisfies the the axioms introduced by Igusa and the author (\cite{Igusa1} and \cite{Axioms}).

\subsection{Geometric Additivity}

In this section we will follow \cite{zbMATH05848700} and show that twisted smooth higher torsion satisfies Igusa's additivity axiom. Consider the following: let $p:E\to B$ be a bundle of compact manifolds and $\F$ a complex local system on $E.$ Assume further that we have bundles $p_1:E_1\to B$ and $p_2:E_2\to B$ so that $E$ is the vertical union $E=E_1\cup E_2$ and they meet along the common vertical boundary $E_1\cap E_2=E_0$ forming a bundle $p_0:E_0\to B.$ We can pull back the local system $\F$ on $E$ along the inclusions to all the $E_i.$ By abuse of notation, we will call the resulting systems again $\F.$ We need to assume that $E_1$ and $E_2$ are unipotent with respect to the local system $\F,$ for $E_0$ this will be guaranteed:

\begin{Lemma} In the situation above, we get that $H_*(F_0;\F)$ is a unipotent $\pi_1 B$-module, where $F_0$ is the fiber of $E_0\to B.$
\end{Lemma}

\begin{proof} Notice that by definition $E_0$ is the vertical boundary of $E_1$ and hence by Poincare duality we get that $H_*(F;\F)\cong H^*(F_1,F_0; \F).$ Since the former is unipotent we get by duality that $H_*(F_1,F_0;\F)$ is unipotent as well. Now the long exact sequence in homology for the pair $(F_1,F_0)$ gives that $H_*(F_0;\F)$ is unipotent as unipotent modules form a Serre-category.
\end{proof} 

\begin{Theorem}[Geometric Additivity] Let $E=E_1\cup E_2$ be a smooth bundle with unipotent local system $\F\to E$ that is the vertical union of the unipotent bundles $E_1$ and $E_2$ (with local systems induced by $\F$) meeting along their common vertical boundary $E_0.$ Then we have
	$$\tau_i(E;\F)=\tau_i(E_1;\F)+\tau_i(E_2;\F)-\tau_i(E_0;\F)\in H^{2i}(B;\Reals).$$
\end{Theorem}

It is well known that $K(\C)$ is an H-space and \cite{zbMATH05551153} gives an explicit H-space structure on $\tilde Q(E_+)$ and shows that the linearization map $\lambda_\F$ is an H-space map. Hence the fiber $Wh_\F(E)$ also is an H-space and the additivity follows directly from the following theorem:

\begin{Theorem} Let $p:E\to B$ be a bundle with a vertical splitting $E=E_1\cup_{E_0} E_2$ as above. Let $j_{i*}:Wh_\F(E_i)\to Wh_\F(E)$ be the maps induced by the inclusions $j_i:E_i\into E.$ Then there exists a homotopy between maps $|\sing B|\to Wh_\F(E)$
	$$\tau_\F(p)\simeq j_{1*}\tau_\F(p_1)+j_{2*}\tau_\F(p_2)+\bar g^{Wh_\F},$$
where $\bar g^{Wh_\F}$ is a map representing the homotopy type of $-j_{0*}\tau_\F(p_0).$
\end{Theorem}

This is essentially proved in \cite{zbMATH05848700}, we will highlight the important steps. First recall that we defined $A(X):=\Omega(|w\T_\bullet\R^{fd}(X)|/|w\T_0\R^{fd}(X))$ where the categories $\T_n\R^{fd}(X)$ have as objects the sequences of cofibrations of retractive spaces over $X$ 
	$$Y_0\cof Y_1\cof\cdots\cof Y_n.$$
So given a manifold bundle $p:E\to B$ we can define a map $p^A:|\sing B|\to A(E)$ induced by the functor $\sing B\to \T_1 \R^{fd}(E)$ given by sending a simplex $\sigma:\Delta^k\to B$ to the cofibration $E\to E\sqcup \sigma^*E.$ By inspection, one can see that the diagram
	$$\xymatrix{
		&	\tilde Q(E_+)\ar[d]\\
		|\sing B|\ar[ur]^{p^!}\ar[r]_{p^A}	&	A(E)}
	$$
commutes. Keeping this notation in mind, we will prove:

\begin{Proposition} Let $E\to B$ be a manifold bundle with a local system $\F$ on $E$ and a vertical splitting as before $E=E_1\cup_{E_0} E_2.$ The inclusions $j_i:E_i\to E$ also induce maps $j_{i*}:\tilde Q((E_i)_+)\to \tilde Q(E_+)$ and $j_{i*}:A(E_i)\to A(E).$ Then we have
	\begin{enumerate}\item There exists a preferred homotopy 
		$$\gamma^A:|\sing B|\times I\to A(E)$$
	between $p^A$ and $j_{1*}p_1^A+j_{2*}p_2^A+g^A,$ where $g^A$ is a map of the homotopy type $-j_{0*}p_0^A.$
		
\item There exists a preferred homotopy 
		$$\gamma^Q:|\sing B|\times I\to \tilde Q(E_+)$$
	between $p^!$ and $j_{1*}p_1^!+j_{2*}p_2^!+g^Q,$ where $g^Q$ is a map of the homotopy type $-j_{0*}p_0^!.$
\end{enumerate}
Furthermore, these constructions fit together well.
\end{Proposition}

The proof of this proposition and the theorem can both be found in \cite{zbMATH05848700} and are (with slight modifications) still applicable in our setting. We will recall the main points.

\begin{proof}[Proof of Proposition] This does not depend on the local system at all, and is therefore applicable without modification. We still recall it to establish some notation. We will only show \textit{i} as \textit{ii} is analogous.

First we take a bicollar neigborhood $E_0\times [-1,1]\to E$ of $E_0$ and we define $E_1'$ and $E_2'$ as the closed complements of this bicollar neighborhood in $E_1$ and $E_2.$ We denote the bundles $q_i:E'_i\to B$ and the inclusions $j'_{i*}:E'_i\to E.$  So it is enough to show
	$$p^A\simeq j'_{1*}q_1^A+j'_{2*}q_2^A+g^A.$$

By inspection it can be seen that the maps $j'_{i*}q_i^A:|\sing B|\to A(E)$ are induced by functors
	$$\overline F_{q_i^A}:\sing B\to w\T_1 \R^{fd}(E)$$
sending a simplex $\sigma:\Delta^n\to B$ to the cofibration $(E\to E\sqcup \sigma^*E'_i).$ Consequently the map $j'_{1*}q_1^A+j'_{2*}q_2^A$ comes from the functor
	$$\overline F_{q_1^A}\sqcup F_{q_i^A}:\sing B\to w\T_1 \R^{fd}(E)$$
sending a simplex $\sigma:\Delta^n\to B$ to the cofibration $(E\cof E\sqcup \sigma^*E'_1\sqcup \sigma^*E'_2).$ Now there is sequence of cofibrations
	$$E\cof E\sqcup (\sigma^*E'_1\sqcup \sigma^*E'_2)\cof E\sqcup \sigma^* E$$
which by Waldhausen's additivity theorem (details in \cite{zbMATH05848700}) induces a homotopy between the map $|\sing B|\to A(E)$ induced by the functor
	$$\sigma\mapsto (E\cof E\sqcup\sigma^*E)$$
and the sum of the maps induced by the functors
	$$\sigma\mapsto (E\cof E\sqcup (\sigma^*E'_1\sqcup \sigma^*E'_2))$$
and
	$$G^A:\sigma\mapsto (E\sqcup (\sigma^*E'_1\sqcup \sigma^*E'_2)\cof E\sqcup \sigma^*E).$$
The first map is none other than $p^A,$ the second is $j'_{1*}q_1^A+j'_{2*}q_2^A$ and the third we will call $g^A.$

Since the cofiber of the cofibration $G^A(\sigma)$ is isomorphic to the object of $\R^{fd}(E)$ obtained by applying the suspension functor to the cofiber of $\overline F_{p_0^A}(\sigma),$ we see that $g^A$ has the homotopy type of $-j_{0*}p_0^A.$
\end{proof}

\begin{proof}[Proof of the Theorem] Using the language established in the previous proof, it is enough to show that
	$$\tau_\F(p)\simeq j'_{1*}\tau_\F(q_1)+j'_{2*}\tau_\F(q_2)+\bar g^{Wh^\F}.$$
The map $\bar g^{Wh^\F}$ can be constructed as follows:

Let 
	$$G^K_\F:\sing B\to w\S_1Ch^{fd}(\Q)$$
be the functor with
	$$G^K_\F(\sigma):=C_*(\sigma^*E,\sigma^*E'_1\sqcup\sigma^*E'_2)$$
and let $g^k:|\sing B|\to A(E)$ be the induced map. By inspection, one can see $g^K=\lambda_\F^A\circ g^A.$ Since $\lambda^A_\F$ is a map of infinite loop spaces and $g^A$ has the homotopy type of $-j_{0*}p_0^A$ we get that $g^K$ has the homotopy type of $-\lambda_\F^Aj_{0*}p_0^A\simeq c_\F.$

If $F'_i$ denotes the fiber of $q_i$ and $F$ denotes the fiber of $p,$ one can construct a homotopy $\omega_g$ from $g^K$ to $*_{H_*(F, F'_1\sqcup F'_2;\F)}$ the same way we did earlier (the $\pi_1 B$-module $H_*(F, F_1'\sqcup F_2';\F)$ is unipotent by the long exact sequence in homology for pairs). Again by inspection one notices $g^K=\lambda_E g^Q$ and therefore the homotopy $\omega_g$ determines a lift
	$$g^{Wh^\F}:|\sing B|\to Wh^\F(E)_{H_*(F, F'_1\sqcup F'_2;\F)}.$$
We denote its reduction by $\bar g^{Wh^\F}:|\sing B|\to Wh^\F(E).$ Since $g^Q$ represents the homotopy type of $-j_{0*}p_0^!$ and that the contracting homotopy $\omega_g$ can be obtained by applying the suspension functor in the category of chain complexes to the construction of the algebraic contraction $\omega_{p_0},$ we see that $\bar g^{Wh^\F}$ has the homotopy type of $-j_0\tau_\F(p_0).$

The rest of the proof can be translated word by word from \cite{zbMATH05848700}.
\end{proof}

\subsection{Geometric Transfer}

Our goal is to show the geometric transfer axiom:

\begin{Theorem}[Geometric Transfer] Let $p:E\to B$ be a unipotent bundle with local system $\F\to E$  and let $q:D\to E$ a linear disc or sphere bundle with fiber $F.$ Then we have
	$$\tau(pq;q^*\F)=\chi(F)\tau(p;\F)+\tr_B^E(\tau(q;q^*\F))$$
where $\chi(F)$ denotes the Euler characteristic and $\tr_B^E$ is the Gottlieb-Becker transfer.
\end{Theorem}	

Again, we follow \cite{zbMATH05848700} closely as all their arguments work just as well with local coefficients. First we make the following definition:

\begin{Definition} Let $q:D\to E$ be a smooth bundle with fiber $F$ and a local system $\F\to E.$ Let $i:F\to D$ be the inclusion of the fiber over the basepoint. Notice that the local system $i^*q^*\F$ is trivial and hence there is a natural map $i^*:H^*(D;q^*\F)\to H^*(F)$ and we call $D\to E$ a Leray-Hirsch bundle if this map $i^*$ has a section $\theta:H^*(F)\to H^*(D;q^*\F)$ with $i^*\theta=\id.$
\end{Definition}

These bundles are exactly the bundles that satisfy the conditions of the Leray-Hirsch Theorem:

\begin{Theorem} Let $q:D\to E$ be a Leray-Hirsch bundle with local system $\F\to E.$ Then there exists an isomorphism
		$$\alpha(q):H_*(D;q^*\F)\to H_*(E;\F)\otimes H_*(F)$$
	and this is natural.
\end{Theorem}

Instead of proving additivity, we will prove the following slightly different Theorem:

\begin{Theorem} The tranfer formula holds for every Leray-Hirsch bundle $q:D\to E$ with fiber $F$ over a unipotent bundle $p:E\to B$ with local system $\F\to E.$
\end{Theorem}

\begin{Remark} This is not strictly stronger than the transfer formula as an odd dimensional linear sphere bundle is not necessarily a Leray-Hirsch bundle. However, even dimensional sphere bundles are Leray-Hirsch bundles. If $S^{2n+1}(\xi)\to E$ is a linear sphere bundle we can split it in an upper and lower hemisphere bundle
	$$S^{2n+1}(\xi)=D^{2n+1}_+(\xi)\cup_{S^{2n}(\xi)}D^{2n+1}_-(\xi)$$
and an application of additivity yields geometric transfer.
\end{Remark}

We begin with the following Lemma.

\begin{Lemma}\label{Transferlemma} Let $q:D\to E$ be a Leray-Hirsch bundle with fiber $F$ and local system $\F\to E.$ Then the following diagram commutes up to a preferred homotopy
	$$\xymatrix{
		\tilde Q(E_+)\ar[r]^{\tilde Q(q^!)}\ar[d]_{\lambda_\F}	&	\tilde Q(D_+)\ar[d]^{\lambda_{q^*\F}}\\
		K(\C)\ar[r]_{\otimes H_*(F)}	&	K(\C).}
	$$
\end{Lemma}

\begin{proof} This can be done following \cite{zbMATH05848700} word by word. The key idea is that the upper right composition is induced by the functor that sends a retractive space $X$ over $E$ to the chain complex $C_*(q^*X,D;q^*\F)$ and the lower left composition is induced by the functor that send the retractive space $X$ to the chain complex $C_*(X,E;\F)\otimes H_*(F).$ The Leray-Hirsch isomorphism for the bundle pair $(q^*X,D)\to (D,E)$ provides the homotopy rendering the diagram commutative.
\end{proof}

Let $q:D\to E$ be a Leray-Hirsch bundle with local system $\F\to E.$ From now on we will refer to the local system $q^*\F\to D$ simply by $\F.$ The homotopy from the Lemma above defines a map $Wh^\F(q^!):Wh^\F(E)\to Wh^\F(D)$ fitting in the diagram  

	$$\xymatrix{
		Wh_\F(E)\ar[d]\ar@{.>}[r]^{Wh^\F(q^!)}						&	Wh_\F(D)\ar[d]\\
		\tilde Q(E_+)\ar[r]^{\tilde Q(q^!)}\ar[d]^{\lambda_\F}	&	\tilde Q(D_+)\ar[d]^{\lambda_\F}\\
		K(\C)\ar[r]^{\otimes H_*(F)}																			&	K(\C).
		}
	$$

\begin{Remark} Notice that in the setting above the map $\otimes H_*(F):K(\C)\to K(\C)$ represents multiplication by $\chi(F)$ in the infinite loop space structure. This gives a commutative diagram
	$$\xymatrix{
		\Omega K(\C)\ar[d]\ar[r]^{\cdot \chi(F)}	&	\Omega K(\C)\ar[d]\\
		Wh^\F(E)\ar[r]^{Wh^\F(q^!)}	&	Wh^\F(D)}.
	$$
The vertical maps are simply induced by the bundle structures.
\end{Remark}

\begin{Proposition} Let $q:D\to E$ a Leray-Hirsch bundle over a unipotent bundle $q:E\to B$ with local system $\F\to E.$ Then the following diagram commutes up to homotopy
	$$\xymatrix{
		&	Wh^\F(E)\ar[d]^{Wh^\F(q^!)}\\
		|\sing B|\ar[ur]^{\tau_\F(p)}\ar[r]_{\tau(pq)}	&	Wh^\F(D)}
	$$
\end{Proposition}

\begin{proof} This again follows \cite{zbMATH05848700}. Here are the key ideas: the homotopy of Lemma \ref{Transferlemma} together with the algebraic contraction of $c_{\F,p}:|\sing B|\to K(\C)$ gives a homotopy between $c_{\F,pq}$ and the constant map with value $H^*(F_p;\F)\otimes H^*(F_q)\in K(\C).$ This homotopy together with the map $(pq)^!$ gives a map
	$$\kappa_\F(pq):|\sing B|\to Wh^\F(D)_{H^*(F_p;\F)\otimes H^*(F_q)}$$
and the reduction of this is homotopic to $Wh^\F(q^!)\circ \tau_\F(q).$ 

The map $\tau_\F(pq)$ is induced by the map $(pq)^!:|\sing B|\to \tilde Q(D_+)$ and the algebraic contraction of $c_{pq}$ to $H_*(F_{pq}).$ Now one can check \cite{zbMATH05848700} that the Leray-Hirsch theorem eventually induces a homotopy rendering the diagram in question commutative.
\end{proof}

The proof of the theorem is now parallel to \cite{zbMATH05848700} with the only difference being that the universal class is $b_{2k}\in H^{2k}(Wh^\F(BG))$ rather than $b_{4k}\in H^{4k}(Wh^\Q(*)).$

\subsection{Transfer of coefficients}

Let $p:E\to B$ and $\tilde p:\tilde E\to B$ be smooth fiber bundles such that $\pi:\tilde E\to E$ is a fiberwise finite covering. Furthermore, let  $\F\to \tilde E$ be a finite unipotent local system corresponding to a representation $\rho:H\to U(m),$ where $H$ is some finite group that the representation of $\pi_1 \tilde E$ factors through. This induces a unipotent local system $\pi_* \F\to E.$ Notice that we can view $H$ as a subgroup of a bigger group $G$ such that the local system $\pi_* \F$ corresponds to the induced representation $\Ind_H^G \rho:G\to U(m).$

\begin{Proposition}[Transfer of coefficients]\label{transfer} In the setting above we have
	$$\tau(\tilde p;\F)=\tau(p;\pi_* \F)\in H^{2k}(B;\R).$$
\end{Proposition}

\begin{proof} The proof is guided by the following diagram:

$$\xymatrix{
	&&	Wh_\F(\tilde E)\ar[rr]\ar[dd]|!{[dl];[dr]}\hole	&&	Wh(BH(N)_+) \ar[dd]\\	
	&	Wh_{\pi_*\F}(E)\ar@{.>}[ur]^{Wh(\pi^*)}\ar[rr]\ar[dd]	&&	Wh(BG(N)_+)\ar@{.>}[ur]^{Wh(\pi^*_{uni})}\ar[dd]\\	
	&& \tilde Q(\tilde E_+)\ar[rr]^(.3){\tilde \phi_Q}|!{[ur];[dr]}\hole\ar[dd]^(.3){\lambda_\F}|!{[dl];[dr]}\hole	&&	\tilde Q(BH(N)_+)\ar[dd]\\	
	&	\tilde Q(E_+)\ar[ur]^{\tilde Q(\pi^*)}\ar[rr]^(.7){\phi_Q}\ar[dd]^(.3){\lambda_{\pi_*\F}}	&&	\tilde Q(BG(N)_+)\ar[ur]^{\tilde Q(\pi^*_{uni})}\ar[dd]\\	
	|\sing B|\ar@{.>}@/^5pc/[uuuurr]^{\tau^{\tilde E}_\F}\ar@{.>}@/^2pc/[uuur]_{\tau^E_{\pi_*\F}}\ar@/^2pc/[uurr]^{\tilde p^!}|!{[uuur];[ur]}\hole\ar[ur]^{p^!}\ar[rr]^(.3){c_\F}|!{[ur];[dr]}\hole \ar[dr]^{c_{\pi_*\F}}	
	&&	K(\C)\ar@{=}[rr]|!{[ur];[dr]}\hole	&&	K(\C)\\	
	&	K(\C)\ar@{=}[ur]\ar@{=}[rr]	&&	K(\C)\ar@{=}[ur]
	}
$$

We first verify that this diagram commutes up to preferred homotopy. The front and backface represent the torsion of $p$ and $\tilde p$ and hence commute up to preferred homotopy.

The bottom left triangle (containing $c_\F$ and $c_{\pi_* \F}$) commutes up to a preferred homotopy since $c_\F$ (and similarly $c_{\pi_*\F}$) is induced by sending a simplex $\sigma:\Delta^n\to B$ to the chain complex $C_*(\sigma^*(\tilde E),\F)$ and there is an isomorphism
	$$H_*(E;\pi_*\F)\cong H_*(\tilde E;\F).$$
	
The map $\tilde Q(\pi^*)$ is defined on $Q(E_+)$ by sending a partition $M$ of $E\times I$ to the partition $(\pi\times \id)^{-1}(M)$ of $\tilde E\times I.$ We can make a similar definition on $A(E)$ and this will define the map $\tilde Q(\pi^*):\tilde Q(E_+)\to \tilde Q(\tilde E_+).$ With this definition the triangle containing $p^!$ and $\tilde p^!$ commutes on the nose. 

To see that the square containing $\lambda_\F$ and $\lambda_{\pi_1\F}$ commutes up to homotopy we first remark that the following square commutes (on the nose)
	$$\xymatrix{
		Q(E_+)\ar[r]^{Q(\pi^*)}\ar[d]	& Q(\tilde E)\ar[d]\\
		A_p(E)\ar[r]	&	A_p(\tilde E)
		}
	$$
which can be seen by comparing the functors inducing those maps. Similarly, the same can be shown if we replace $Q$ with $A$ and hence it is enough to see that the square 
	$$\xymatrix{
		A(E_+)\ar[r]^{A(\pi^*)}\ar[d]_{\lambda_{\pi_*\F}}	&	A(\tilde E_+)\ar[d]^{\lambda_\F}\\
		K(\C)\ar@{=}[r]	&	K(\C)
	}
	$$
commutes up to homotopy. Again, this is because the top right composition comes from the functor sending a retractive space $X$ over $E$ to the chain complex $C_*(\pi^* X, \tilde E;\F)$ and the bottom left composition would send it to $C_*(X,E;\pi_*\F)$ and both complexes have the same homology. This homotopy commutativity induces the map $Wh(\pi^*).$

Since the local system $\F\to \tilde E$ is unipotent we know that both maps $c_\F$ and $c_{\pi_*\F}$ are homotopy contractible, giving the lifts $p^{Wh}$ and $\tilde p^{Wh}.$ We will verify in the next Lemma that all homotopies used are compatible which guarantees that the triangle containing $p^{Wh}$ and $\tilde p^{Wh}$ also commutes up to preferred homtopy.

Notice that there is a covering $\pi_{uni}:BH(N)\to BG(N)$ so for the same reason as above the square on the right commutes up to homotopy. Since the maps $\phi_Q$ and $\tilde \phi_Q$ are induced by actual maps of spaces $E\to BG(N)$ and $\tilde E\to BH(N)$ (themselves induced by the local system) it is clear that the middle horizontal square commutes. 

Hence the whole diagram is homotopy commutative. We know that the torsion of $p$ pulls back from a universal class $b_k^G\in H^{2k}(Wh(BG(N));\Reals)$ and the torsion of $\tilde p$ pulls back from the universal class $b_k^H\in H^{2k}(Wh(BH(N));\Reals).$ So all we need to do is to show
	$$Wh(\pi^*_{uni})^*b^H_k=b^G_k.$$

This follows since after stabilization and functorial cofibrant replacement the classes correspond to classes $b_k^G\in H^{2k}(Wh(BG)_\Q;\Reals)$ and $b_k^H\in H^{2k}(Wh(BH)_\Q;\Reals).$ These classes are pulled back from a universal class $b_k\in \Omega K(\C)_\Q$ via the diagram
	$$\xymatrix{
		&	Wh(BH)_\Q\ar[dd]|!{[dl];[dr]}\hole\ar[dr]\\
		Wh(BG)_\Q\ar[ur]\ar[rr]\ar[dd]	&&	\Omega K(\C)_\Q\ar[dd]\\
		&	\tilde Q(BH)_\Q\ar[dr]^\simeq\ar[dd]|!{[dl];[dr]}\hole \\
		\tilde Q(BG)_\Q\ar[dd]\ar[ur]\ar[rr]^(.7)\simeq	&&	\ast_\Q\ar[dd]\\
		&	K(\C)_\Q\ar@{=}[dr]\\
		K(\C)_\Q\ar@{=}[ur]\ar@{=}[rr] && K(\C)_\Q
		}
	$$
Since the contractions of $\tilde Q(BG)$ and $\tilde Q(BH)$ are compatible ($BH\to BG$ is a finite covering space) this diagram commutes up to homotopy and finishes the proof.

\end{proof}

\begin{Lemma}\label{compatibel} The homotopies used in in the proof of Proposition \ref{transfer} are compatible, meaning that there exists a preferred homotopy $Wh(\pi^*)\tau^E_{\pi_*\F}\simeq \tau^{\tilde E}_\F.$
\end{Lemma}

\begin{proof} As laid out in \cite{zbMATH05848700} we need to provide homotopies $\tilde Q(\pi^*)p^!\simeq \tilde p^!,$ $\lambda_{\pi_*\F}\simeq \lambda_\F\tilde Q(\pi^*),$ a path $\gamma$ in $K(\C)$ connecting $*_{H_*(F_{b_0},\pi_*\F)}$ and $*_{H_*(\tilde F_{b_0};\F)}$ (where $b_0\in B$ is the base point) and a homotopy of homotopies filling the diagram (where the corners are maps $|\sing B|\to K(\C)$ and the arrows are homotopies) 
	$$\xymatrix{
		p^!\lambda_{\pi_*\F}\ar[r]\ar[d]^\omega	&	\tilde p^!\lambda_\F\ar[d]^{\omega'}\\
		(pt)_{H_*(F_{b_0},\pi_*\F)}\ar[r]^\gamma	&		(pt)_{H_*(\tilde F_{b_0};\F)}.
		}
	$$
The horizontal homotopies were already given in the proof of the proposition. Recall that they are induced by the canoncial isomorphism 
	$$H_*(E;\pi_*\F)\cong H_*(\tilde E;\F).$$
The path $\gamma$ will be induced by the canonical isomorphism
	$$H_*(F_{b_0};\pi_*\F)\cong H_*(\tilde F_{b_0};\F).$$
The homotopies $\omega$ and $\omega'$ were constructed in the proof of theorem \ref{contraction}. Recall that they were constructed as the concatenation of three homotopies. By analyzing each of those three homotopies, the square above splits in three squares, where the vertical homotopies are induced similarly to how $\gamma$ was induced. Hence it is easy to see that all of those squares will commute (on the nose).
\end{proof}

\subsection{Additivity of coefficients}

Let $p:E\to B$ be a smooth bundle and $\F_1\to E$ and $\F_2\to E$ be two unipotent local system on $E$ corresponding to representations $G_1\to U(m_1)$ and $G_2\to U(m_2).$ Then the sum $\F_1\oplus \F_2\to E$ is a unipotent local system on $E$ corresponding to the representation $\rho_1\oplus \rho_2:G_1\times G_2\to U(m_1+m_2).$ We aim to prove

\begin{Proposition}[Additivity of coefficients] In the setting above we have 
	$$\tau_k(p;\F_1)+\tau_k(p;\F_2)=\tau(p;\F_1\oplus \F_2)\in H^{2k}(B;\Reals)$$
for any $k\in\mathbb{N}.$
\end{Proposition}

\begin{proof} In this proof only, we will rather work with the representations $\rho_i$ corresponding to the local systems $\F_i$ to emphasize how the concrete model of $BG_i$ depends on the dimension of the representation. The proof is guided by the following diagram:

{\tiny
$$\xymatrix{
	&&	Wh'_{\rho_1,\rho_2}(E)\ar@{.>}[dl]\ar[rr]\ar[dd]|!{[dl];[dr]}\hole	&&	Wh(BG_1(N))\times Wh(BG_2(N))\ar[dl]\ar[dd]\\
	&	Wh_{\rho_1\oplus \rho_2}(E)\ar[rr]\ar[dd]	&&	Wh(BG_1(N)\times BG_2(N))\ar[dd]\\
	&& \tilde Q(E_+)\ar[rr]^(.3){\tilde Q(\rho_1)\times \tilde Q(\rho_2)}|!{[ur];[dr]}\hole\ar[dd]^(.3){\lambda_{\rho_1}\times \lambda_{\rho_2}}|!{[dl];[dr]}\hole && \tilde Q(BG_1(N))\times \tilde Q(BG_2(N))\ar[dd]\ar[dl]^\cong\\
	& \tilde Q(E_+)\ar@{=}[ur]\ar[rr]_(.3){\tilde Q(\rho_1\oplus\rho_2)}\ar[dd]^(.3){\lambda_{\rho_1\oplus\rho_2}}	&&	\tilde Q(B(G_1(N)\times G_2(N)))\ar[dd]\\
	|\sing B|\ar@{.>}@/^5.5pc/[uuuurr]^{p_{\rho_1,\rho_2}^{Wh}}\ar@{.>}@/^2pc/[uuur]_{\tau_{\F_1\oplus\F_2}}\ar@/^2pc/[uurr]^(.3){ p^!}|!{[uuur];[ur]}\hole\ar[ur]^{p^!}\ar[rr]^(.3){c_{\rho_1}\times c_{\rho_2}}|!{[ur];[dr]}\hole \ar[dr]^{c_{\rho_1\oplus \rho_2}}
	&&	K(\C)\times K(\C)\ar@{=}[rr]|!{[ur];[dr]}\hole\ar[dl]^\oplus	&&	K(\C)\times K(\C)\ar[dl]^\oplus\\
	&	K(\C)\ar@{=}[rr]	&&	K(\C)
	}
$$}
	
We first establish that this diagram commutes up to preferred homtopy. For the front face this is clear since it simply represents the torsion $\tau(E,\F_1\oplus \F_2).$ The space $Wh'_{\rho_1,\rho_2}(E)$ simply denotes the homotopy fiber (over the base point) of $\lambda_{\rho_1}\times \lambda_{\rho_2}:\tilde Q(E_+)\to K(\C)\times K(\C),$ and with this definition it is clear that the back face also commutes up to preferred homotopy. All of the triangles and squares on the very left clearly commute up to a preferred homotopy because there is an isomorphism
	$$H_n(E,\rho_1\oplus \rho_2)\cong H_n(E,\rho_1)\oplus H_n(E,\rho_2).$$
It can be verified that this behaves well with respect to the contracting homotopies and therefore gives rise to commutativity up to homotopy on the level of the homotopy fibers. This is done by the same strategy as in Lemma \ref{compatibel}, just that in this case the homotopies are all induced by the isomorphism $H_*(E;\rho_1\oplus\rho_2)\cong H_*(E;\rho_1)\oplus H_*(E;\rho_2).$ The same arguments apply to the squares on the right hand side.

Next we analyze the origin of the higher torsion. After stabilization and rationalization we have the following commutative (up to preferred homotopy) diagram

$$\xymatrix{
	& Wh(BG_1)_\Q\times WH(BG_2)_\Q\ar[rr]\ar[dl]\ar[dd]|!{[dl];[dr]}\hole	&&	\Omega K(\C)_\Q\times \Omega K(\C)_\Q\ar[dl]^\oplus\ar[dd]\\
	Wh(B(G_1\times G_2))_\Q\ar[rr]\ar[dd]	&&	\Omega K(\C)_\Q\ar[dd]\\
	& \tilde Q(BG_1)_\Q\times \tilde Q(BG_2)_\Q\ar[rr]|!{[ur];[dr]}\hole\ar[dl]^\cong\ar[dd]|!{[dl];[dr]}\hole	&&	\ast_\Q\ar[dd]\ar@{=}[dl]\\
	\tilde Q(B(G_1\times G_2))_\Q\ar[dd]\ar[rr]	&&	\ast_\Q\ar[dd]\\
	&	K(\C)_\Q\times K(\C)_\Q\ar@{=}[rr]|!{[ur];[dr]}\hole\ar[dl]^\oplus	&&	K(\C)\times K(\C)\ar[dl]^\oplus\\
	K(\C)\ar@{=}[rr] &&	K(\C).
	}
$$
Recall that the higher torsion is pulled back from the Borel regulator $b_k\in H^{2k}(\Omega K(\C);\Reals).$ We will verify in a Lemma below that under the map $\oplus:\Omega K(\C)\times \Omega K(\C)\to \Omega K(\C)$ the class $b_k$ pulls back exactly to
	$$b_k\otimes 1+1\otimes b_k\in H^{2k}(\Omega K(\C)\times \Omega K(\C);\Reals).$$
Therefore we will now investigate what this element pulls back to along the back face of the first diagram. 

For this consider the following diagram

$$\xymatrix{
	&	Wh'_{\rho_1,\rho_2}(E)\ar[d]\ar[r]	&	Wh_{\rho_1}(E)\times Wh_{\rho_2}(E)\ar[d]\ar[r]	&	Wh(BG_1(N))\times Wh(BG_2(N))\ar[d]\\
	&	\tilde Q(E_+)\ar[r]^\Delta\ar[d]\ar[d]^(.3){(\lambda_{\rho_1},\lambda_{\rho_2})}	&	\tilde Q(E_+)\times \tilde Q(E_+)\ar[d]^{\lambda_{\rho_1}\times \lambda_{\rho_2}}\ar[r]	&	\tilde Q(BG_1(N))\times \tilde Q(BG_2(N))\ar[d]\\
	|\sing B|\ar[ur]^{p^!}\ar[urr]|!{[ur];[r]}\hole_(.6){(p^!,p^!)}\ar[r]_{c_{\rho_1\oplus \rho_2}}\ar[uur]\ar@/^6pc/[uurr]^{(\tau_{\F_1},\tau_{\F_2})}	&	K(\C)\times K(\C)\ar@{=}[r]	&	K(\C)\times K(\C)\ar@{=}[r]	&	K(\C)\times K(\C)}
$$
It is clear that this diagram commutes and the upper horizontal composition is exactly the upper horizontal composition on the back face of the original diagram. This shows that the element $b_k\otimes 1+1\otimes b_k\in H^{2k}(\Omega K(\C)\times \Omega K(\C);\Reals)$ pulls back along the backface of the first diagram to 
	$$\tau(E,\F_1)+\tau (E,\F_2)\in H^{2k}(B;\Reals)$$
and this completes the proof.

\end{proof}

\begin{Lemma} Let $X$ be a unital H-space and $\oplus:\Omega X\times \Omega X\to \Omega X$ be induced by the H-space map. Furthermore, let $R$ be a ring and $b\in \tilde H^k(X;R)$ be represented by the map $f:X\to K(R,k)$ and $\tau(b)\in \tilde H^{k-1}(X;R)\subset H^{k-1}(X;R)$ be represented by the map $\Omega f:\Omega X\to K(R,k-1).$ Then the pull back along $\oplus$ of $\tau(b)$ gives
	$$\oplus^*\tau(b)=\tau(b)\otimes 1+1\otimes \tau(b)\in H^{k-1}(\Omega X\times \Omega X).$$
\end{Lemma}

\begin{proof} By the Eckmann-Hilton argument we know that the H-space structure on $\Omega X$ induced by $\oplus$ is equivalent to the H-space structure given by concatenation of loops $c:\Omega X\times \Omega X\to \Omega X.$ It is well known (see for example Proposition 16.19 in \cite{zbMATH01716466}) that this satisfies
	$$c^*\tau(b)=\tau(b)\otimes 1 +1\otimes \tau(b).$$
\end{proof}

\bibliography{mybib}{}
\bibliographystyle{plain}

\end{document}